\documentclass[11pt, amsfonts]{amsart}

\usepackage{amsmath,amssymb,amsthm}
\usepackage[all]{xy}
\usepackage[dvipdfm,colorlinks=true]{hyperref}

\numberwithin{equation}{section}

\SelectTips{eu}{12}

\newtheorem{theorem}{Theorem}[section]
\newtheorem{proposition}[theorem]{Proposition}
\newtheorem{lemma}[theorem]{Lemma}
\newtheorem{corollary}[theorem]{Corollary}

\theoremstyle{definition}

\theoremstyle{remark}

\newcommand{\Z}{\mathbb{Z}}
\newcommand{\g}{\mathcal{G}}
\newcommand{\map}{\operatorname{Map}}

\title[The mod-$p$ homology of the classifying spaces of certain gauge groups]{The mod-$p$ homology of the classifying spaces of certain gauge groups}

\author{Daisuke Kishimoto}
\address{Department of Mathematics, Kyoto University, Kyoto, 606-8502, Japan}
\email{kishi@math.kyoto-u.ac.jp}

\author{Stephen Theriault}
\address{Mathematical Sciences, University of Southampton, Southampton SO17 1BJ, United Kingdom}
\email{s.d.theriault@soton.ac.uk}


\begin{document}

\baselineskip.525cm

\maketitle

\begin{abstract}
  Let $G$ be a compact connected simple Lie group of type $(n_{1},\ldots,n_{l})$, where $n_{1}\le\cdots \le n_{l}$. Let $\g_k$ be the gauge group of the principal $G$-bundle over $S^4$ corresponding to $k\in\pi_3(G)\cong\Z$. We calculate the mod-$p$ homology of the classifying space $B\g_k$ provided that $n_{l}<p-1$. %We also calculate the mod-$3$ homology of $B\g_k$ when $G=SU(2)$ and $k$ is not divisible by $3$.
\end{abstract}

%%%%% Section 1 %%%%%

\section{Introduction}
\label{sec:intro}

Let $G$ be a Lie group and $P\to X$ be a principal $G$-bundle over a manifold $X$. Automorphisms of $P$ are by definition $G$-equivariant self-maps of $P$ covering the identity map of $X$. The topological group of automorphisms of $P$ is called the \emph{gauge group} of $P$; we denote this group by $\g(P)$. The classifying space of $\g(P)$ is denoted by $B\g(P)$.

Gauge groups are fundamental in modern physics and geometry. Since the classifying space $B\g(P)$ is homotopy equivalent to the moduli space of connections on $P$ as in \cite{AB}, the topology of gauge groups over $4$-manifolds and their classifying spaces has proved to be of immense value in studying diffeomorphism structures on $4$-manifolds \cite{D}, Yang-Mills theory \cite{AJ}, and invariants of $3$-manifolds \cite{F}. An important problem is to calculate the mod-$p$ homology of $B\g(P)$ for a prime $p$ when the base space $X$ of $P$ is a closed $4$-manifold, as this would give rise to characteristic classes that could be used in geometric and physical contexts. A certain subring of the mod-$2$ cohomology of $B\g(P)$ was studied by Masbaum \cite{Ma} when $G=SU(2)$ and $X$ is a simply-connected closed $4$-manifold, and Choi \cite{C} has some partial results on the mod-$2$ homology for $G=Sp(n)$ and $X=S^4$. However, there is no complete calculation of the mod-$2$ homology or cohomology of $B\g(P)$ when $X$ is a closed $4$-manifold. As for the mod-$p$ homology or cohomology for odd primes, practically nothing was known to this point.

Let $G$ be a compact connected simple Lie group. In this paper we calculate the mod-$p$ homology of the classifying spaces of gauge groups of principal $G$-bundles over $S^4$ under a certain condition on the prime $p$. To state our results, we need some notation. Principal $G$-bundles over $S^{4}$ are classified by $\pi_3(G)$, where $\pi_3(G)\cong\Z$ since $G$ is simple. Let $\g_k$ denote the gauge group of a principal $G$-bundle over $S^4$ corresponding to $k\in\Z\cong\pi_3(G)$. Let $\map_{k}(S^{4},BG)$ be the component of the space of continuous (not necessarily pointed) maps from $S^{4}$ to $BG$ which are of degree $k$, and similarly define $\map^*_{k}(S^{4},BG)$ with respect to pointed maps. There is a fibration
\begin{equation}
\label{eval_k}
\map_k^*(S^4,BG)\to \map_k(S^4,BG)\xrightarrow{ev}BG
\end{equation}
where $ev$ evaluates a map at the basepoint of $S^{4}$. Let $G\langle 3\rangle$ be the three-connected cover of $G$. For each $k\in\mathbb{Z}$, the space $\map^*_{k}(S^{4},BG)$ is homotopy equivalent to $\Omega^{3} G\langle 3\rangle$. By \cite{AB,G}, there is a homotopy equivalence $B\g_k\simeq\map_{k}(S^{4},BG)$. Thus there is a homotopy fibration
\begin{equation}
  \label{Serrefib}
  \Omega^{3} G\langle 3\rangle\to B\g_k\xrightarrow{ev}{BG}.
\end{equation}

The Lie group $G$ has \emph{type} $(n_{1},\ldots,n_{l})$ if the rational cohomology of $G$ is generated by elements in degrees $2n_{1}-1,\ldots,2n_{l}-1$, where $n_{1}\le\cdots \le n_{l}$. Unless otherwise indicated, homology is assumed to be with mod-$p$ coefficients.

\begin{theorem}
   \label{Bgaugehlgy}
   Let $G$ be a compact connected simple Lie group of type $(n_{1},\ldots,n_{l})$ and let $p$ be a prime. If $n_{l}<p-1$ then there is an isomorphisms of $\mathbb{Z}/p\mathbb{Z}$-vector spaces
   $$H_*(B\g_k)\cong H_*(BG)\otimes H_*(\Omega^3G\langle 3\rangle).$$
\end{theorem}

Under the assumption of Theorem \ref{Bgaugehlgy} the Lie group $G$ is $p$-locally homotopy equivalent to the product $\prod_{i=1}^l S^{2n_i-1}$, so we can calculate the Poincar\'e series of $H_*(B\g_k)$ by Theorem \ref{Bgaugehlgy} (Corollary \ref{Poincare_series}). It is remarkable that Theorem \ref{Bgaugehlgy} implies that the mod-$p$ homology of $B\g_k$ is independent from $k$ for $p$ large, whereas there are more than one $p$-local homotopy types in the family $\{B\g_k\}_{k\in\Z}$ for $p$ large as was proved in \cite{KT1}. The approach to Theorem \ref{Bgaugehlgy} is to consider the Serre spectral sequence applied to the fiberwise coproduct of \eqref{Serrefib}. Control is obtained over the differentials by showing that the first nontrivial differential is a transgression on a certain element, which is not obvious, and then atomicity-style arguments (cf. \cite{S}) are used to show the spectral sequence must collapse at the $E^{2}$-term.

A notable example is when $G=SU(2)$, a case of particular interest in Donaldson Theory and $3$-manifold theory. Theorem \ref{Bgaugehlgy} holds if $p\geq 5$ in this case. Let $(m,n)$ denote the gcd of integers $m$ and $n$. We go on to calculate the mod-$3$ homology of $B\g_k$ for $(k,3)=1$, where the case $(k,3)=3$ is still open. For $(k,3)=1$, it is the Serre spectral sequence for the homotopy fibration $SU(2)\to\Omega^3SU(2)\langle 3\rangle\to B\g_k$ induced from \eqref{Serrefib} that collapses at the $E^{2}$-term.

\begin{theorem}
   \label{p=3}
   Let $G=SU(2)$ and $p=3$. If $(k,3)=1$ then there is an isomorphism of $\Z/3\Z$-vector spaces
   $$H_*(B\g_k)\cong H_*(\Omega^{3} S^{3}\langle 3\rangle)/(x_{3})$$
   where $x_3$ is a generator of $H_3(\Omega^{3} S^{3}\langle 3\rangle)\cong\Z/3\Z$ and $(x_{3})$ is the ideal generated by $x_{3}$.
\end{theorem}

The difference of the mod-$p$ homology of $B\g_k$ for $G=SU(2)$ in Theorems \ref{Bgaugehlgy} and \ref{p=3} comes from the homotopy commutativity of $SU(2)$; it is homotopy commutative if and only if $p\ge 5$ as in \cite{Mc}. This is notable because the product decomposition $\g_k\simeq G\times\Omega^4G\langle 3\rangle$ as $A_n$-spaces is guaranteed by the higher homotopy commutativity of $G$ as in \cite{KiKo,KT1}, whereas Theorem \ref{Bgaugehlgy} shows the homological product decomposition as $A_\infty$-spaces.

\emph{Acknowledgement:} The first author was partly supported by JSPS KAKENHI 17K05248.

%%%%% Section 2 %%%%%

\section{Serre spectral sequence}

Consider a homotopy fibration
\begin{equation}
  \label{fibration action}
  F\to E\to B
\end{equation}
over a path-connected base $B$ such that $F$ is an H-space and there is a fiberwise action of $F$ on $E$ which restricts to the multiplication of $F$. We assume that $\pi_1(B)$ acts trivially on $H_*(F)$. Let $(E^r,d^r)$ denote the associated homology Serre spectral sequence.

\begin{lemma}
  \label{Leibniz rule}
  There is a coalgebra map
  \[
    \mu\colon E^r_{p,q}\otimes H_{q'}(F)\to E^r_{p,q+q'}
  \]
  having the following properties.
  \begin{enumerate}
    \item The map
    \[
      \mu\colon H_p(B)\otimes H_q(F)=E^2_{p,0}\otimes H_q(F)\to E^2_{p,q}
    \]
    coincides with the canonical isomorphism;

    \item For $x\in E^r$ and $y\in H_*(F)$,
    \[
      d^r(\mu(x\otimes y))=\mu(d^r(x)\otimes y).
    \]
  \end{enumerate}
\end{lemma}

\begin{proof}
  Since the action of $F$ on $E$ is fiberwise, there is a homotopy commutative diagram
  \[
    \xymatrix{
      F\times F\ar[r]\ar[d]&F\times E\ar[r]\ar[d]&B\ar@{=}[d]\\
      F\ar[r]&E\ar[r]&B.
    }
  \]
  Since rows of this diagram are homotopy fibrations, there is a map between associated homology Serre spectral sequences. Since the homology Serre spectral sequences of the top row is isomorphic with $(E^r\otimes H_*(F),d^{r}\otimes 1)$, we obtain the map $\mu$. By definition, $\mu$ is a coalgebra map. The second statement holds because the differential on $H_*(F)$ in $E^r\otimes H_*(F)$ is trivial. Let $p\colon E\to B$ denote the projection. For each $y\in E$, $\pi^{-1}(\pi(y))$ is homotopy equivalent to the orbit space $y\cdot F$ including $y$. Therefore, by the construction of the Serre spectral sequence, the first statement holds.
\end{proof}

\begin{corollary}
  \label{cor Leibniz rule}
  If $d^2=\cdots=d^{r-1}$, then the map $\mu\colon E^r_{p,0}\otimes H_q(F)\to E^r_{p,q}$ coincides with the canonical isomorphism
  \[
    E^r_{p,q}\cong H_p(B)\otimes H_q(F).
  \]
\end{corollary}

Let $\overline{F}\to\overline{E}\to B$ be a homotopy fibration which is a homotopy retract of \eqref{fibration action}. Let $(\overline{E}^r,\overline{d}^r)$ denote the associated homology Serre spectral sequence.

\begin{lemma}
  \label{retract SS}
  If $\overline{d}^2=\cdots=\overline{d}^{r-1}=0$ and $\overline{d}^r\ne 0$, then the following statements hold:
  \begin{enumerate}
    \item $\overline{d}^rx\ne 0$ for some $x\in H_*(B)=\overline{E}^r_{*,0}$;

    \item If $y$ is an element of least degree in $H_*(B)=\overline{E}^r_{*,0}$ with $\overline{d}^ry\ne 0$, then $y$ is transgressive and $\overline{d}^r(y)$ is a primitive element of $H_{*-1}(F)=\widehat{E}^r_{0,*-1}$.
  \end{enumerate}
\end{lemma}

\begin{proof}
  Let $i\colon\overline{E}^2\to E^2$ and $q\colon E^2\to\overline{E}^2$ denote the inclusion and the retraction, respectively. Since
  \[
    d^2(x)=d^2(i_*(x))=i_*(\overline{d}^2(x))=0
  \]
the hypothesis that $\overline{d}^{2}=0$ implies that $d^{2}=0$. Therefore $\overline{E}^{3}=\overline{E}^{2}$ is a retract of $E^{3}=E^{2}$. Iterating this argument, since $\overline{d}_{3}=\cdots=\overline{d}^{r-1}=0$, we also obtain $d^{2}=\cdots=d^{r-1}=0$ and therefore $\overline{E}^{r}$ is a retract of $E^{r}$. The inclusion and retraction at the $r^{th}$-stage  may still be denoted by $i\colon\overline{E}^2\to E^2$ and $q\colon E^2\to\overline{E}^2$. Suppose that $\overline{d}^r(x)=0$ for all $x\in H_p(B)=\widehat{E}^r_{p,0}$. Then
  \[
    d^r(x)=d^r(i_*(x))=i_*(\overline{d}^r(x))=0,
  \]
  implying $d^r=0$ by Lemma \ref{Leibniz rule} and Corollary \ref{cor Leibniz rule}. This is a contradiction, so the first statement is proved.

  Let $y$ be an element of least degree element in $H_*(B)=\overline{E}^r_{*,0}$ such that $\overline{d}^ry\ne 0$. Let $\Delta$ denote comultiplication, and set
  \[
    \Delta(y)=y\otimes 1+1\otimes y+\sum_iy_i'\otimes y_i''
  \]
  where $|y_i'|<|y|$ and $|y_i''|<|y|$ for each $i$. Since $y$ is an element of least degree in $H_*(B)=E^r_{*,0}$ with $d^ry\ne 0$, we have $d^ry_i'=0$ and $d^ry_i''=0$ for each $i$. Therefore
  \[
    \Delta(d^r(y))=d^r(\Delta(y))=d^r\left(y\otimes 1+1\otimes y+\sum_iy_i'\otimes y_i''\right)=d^r(y)\otimes 1+1\otimes d^r(y)
  \]
  so $d^{r}(y)$ is primitive.

  If $r<|y|$, then by Corollary \ref{cor Leibniz rule}, $d^r(y)=\mu(a\otimes b)$ for some non-trivial $a\in H_{|y|-r}(B)$ and $b\in H_{r-1}(F)$. This is impossible because $\mu(a\otimes b)$ is not primitive by Lemma \ref{Leibniz rule}. Thus $r\ge |y|$. Clearly, $r\le |y|$ since $r>|y|$ implies $d^{r}(y)$ lands in the second quadrant of the spectral sequence, which is zero. Therefore $r=|y|$, implying that $y$ is transgressive. Moreover, since
  \[
    \overline{d}^r(y)=\overline{d}^r(q_*(y))=q_*(d^r(y))
  \]
  and $d^ry$ is primitive, $\overline{d}^ry$ is also primitive. Therefore the second statement is proved.
\end{proof}

Next, we consider a family of homotopy fibrations $F_n\to E_n\to B$ with a common base $B$ for $n\in\Z$. Let $(E^r_n,d^{r}_n)$ denote the associated homology Serre spectral sequence. We can form a homotopy fibration
\[
  \coprod_{n\in\Z}F_n\to\coprod_{n\in\Z}E_n\to B.
\]
Let $(\widehat{E}^r,\widehat{d}^{r})$ denote the associated homology spectral sequence.

\begin{lemma}
  \label{collapse component}
  If $(\widehat{E}^r,\widehat{d}^{r})$ collapses at the second term, then so does $(E^r_n,d^{r}_n)$ for each $n\in\Z$.
\end{lemma}

\begin{proof}
  Since there are isomorphisms
  \[
    (E^2_n)_{p,q}\cong H_p(B)\otimes H_q(F_n)\quad\text{and}\quad\widehat{E}^2_{p,q}\cong H_p(B)\otimes\left(\bigoplus_{n\in\Z}H_q(F_n)\right),
  \]
  the inclusion $E_n\to E$ induces an injection $E^2_n\to\widehat{E}^2$. Then the statement is proved by induction on $r$.
\end{proof}

%%%%% Section 3 %%%%%

\section{The mod-$p$ homology of $\Omega^{3} G\langle 3\rangle$ when $G$ is $p$-regular}
\label{sec:3connhlgy}

Localize at an odd prime $p$ and take homology with mod-$p$ coefficients. If $G$ is $p$-regular of type $(n_{1},\ldots, n_{l})$ then there is a homotopy equivalence $G\simeq\prod_{i=1}^{l} S^{2n_{i}-1}$. Therefore
\begin{equation}
  \label{product_decomp}
  \Omega^{3} G\langle 3\rangle\simeq\Omega^{3} S^{3}\langle 3\rangle\times\prod_{i=2}^{l}\Omega^{3} S^{2n_{i}-1}.
\end{equation}
This is an equivalence of H-spaces and so induces an isomorphism of Hopf algebras in homology. In this section we record a property of $\Omega^{3} G\langle 3\rangle$ which will be important later. This begins with a general definition.

In general, for a path-connected space $X$ of finite type, let $PH_*(X)$ be the subspace of primitive elements in $H_*(X)$. Let
$$\mathcal{P}^{n}_*\colon H_{q}(X)\to H_{q-2(p-1)n}(X)$$
be the dual of the Steenrod operation $\mathcal{P}^{n}$. For $r\geq 1$, let $\beta^{r}$ be the $r^\text{th}$-Bockstein. Let $MH_*(X)$ be the subspace of $PH_*(X)$ defined by
$$MH_*(X)=\{x\in PH_*(X)\mid\,\mathcal{P}^{n}_*(x)=0\text{ for all }n>0\text{ and }\beta^{r}(x)=0\text{ for all }r\geq 1\}.$$
Since $MH_*(X)$ records information about the primitive elements in $H_*(X)$, if $X\simeq A\times B$ then
$$MH_*(X)=MH_*(A\times B)=MH_*(A)\oplus MH_*(B).$$

Now consider $MH_*(\Omega^{3} G\langle 3\rangle)$. The product decomposition \eqref{product_decomp} implies that
\begin{equation}
  \label{Msplit}
  MH_*(\Omega^{3}\langle G\rangle)=MH_*(\Omega^{3} S^{3}\langle 3\rangle)\oplus\left(\bigoplus_{i=2}^{l} MH_*(\Omega^{3} S^{2n_{i}-1})\right).
\end{equation}
By \cite{CLM} there is an isomorphism of Hopf algebras
\begin{multline}
  \label{CLMcalc}
  H_*(\Omega^{3} S^{2n+1})\cong\bigotimes_{k\geq 1, j\geq 0}\Lambda(a_{2(np^{k}-1)p^{j}-1})\\
  \otimes\bigotimes_{k\geq 1, j\geq 1}\Z/p\Z[b_{2(np^{k}-1)p^{j}-2}]\otimes\bigotimes_{k\geq 0}\Z/p\Z[c_{2n^{k}-2}]
\end{multline}
such that $|a_i|=|b_i|=i$. Here, the generators are primitive and many are related by the action of the dual Steenrod algebra. Selick \cite{S} determined $MH_*(\Omega^{3} S^{2n+1})$ in full, we record only the subset of elements of odd degree.

\begin{lemma}
   \label{MHsphere}
   If $n>1$ then $MH_\text{odd}(\Omega^{3} S^{2n+1})=\Z/p\Z\{a_{2np-3}\}$.
\end{lemma}

In the references that follow for $\Omega^{3} S^{3}\langle 3\rangle$, the statements in \cite{Th} are in terms of Anick spaces, but by \cite{GT} the space $\Omega S^{3}\langle 3\rangle$ is homotopy equivalent to the Anick space $T^{2p+1}(p)$ for $p\geq 3$. By \cite[Proposition 4.1]{Th}, for $p$ odd there is an isomorphism of Hopf algebras
\begin{equation}
   \label{Thlgy}
   H_*(\Omega^{3} S^{3}\langle 3\rangle)\cong H_*(\Omega^{2} S^{2p-1})\otimes H_*(\Omega^{3} S^{2p+1})
\end{equation}
which respects the action of the dual Steenrod operations and the Bockstein operations. This can be phrased in terms of a generating set using the isomorphism of Hopf algebras
\begin{equation}
  \label{CLMcalc_2p-1}
  H_*(\Omega^{2} S^{2p-1})\cong\bigotimes_{k=0}^{\infty}\Lambda(\bar{a}_{2(p-1)p^{k}-1})\otimes\bigotimes_{k=1}^{\infty}\mathbb{Z}/p\mathbb{Z}[\bar{b}_{2(p-1)p^{k}-2}]
\end{equation}
proved in \cite{CLM}, where $|\bar{a}_i|=|\bar{b}_i|=i$, and the $n=p$ case of \eqref{CLMcalc}. Again, the generators are primitive and many are related by the action of the dual Steenrod algebra. A description of $MH_*(\Omega^{3}\langle 3\rangle)$ in full was given in \mbox{\cite[Lemma 4.2]{Th}}, but again we need only record the subset of elements of odd degree.

\begin{lemma}
   \label{MHT}
   $MH_\text{odd}(\Omega^{3} S^{3}\langle 3\rangle)=\Z/p\Z\{\bar{a}_{2p-3}\}$.
\end{lemma}

Combining \eqref{Msplit}, Lemmas \ref{MHsphere} and \ref{MHT} we obtain
the following.

\begin{lemma}
   \label{MHG}
   If $G$ is $p$-regular of type $(n_{1},\ldots,n_{l})$ then
   $$MH_\text{odd}(\Omega^{3} G\langle 3\rangle)=\Z/p\Z\{\bar{a}_{2p-3},a_{2n_{2}p-3},\ldots,a_{2n_{l}p-3}\}.$$
\end{lemma}

\section{The proof of Theorem \ref{Bgaugehlgy}}
\label{sec:homology}

We continue to localize at an odd prime $p$ and take homology with mod-$p$ coefficients. Let $G$ be a compact connected simple Lie group of type $(n_1,\ldots,n_l)$. Consider the homotopy fibration sequence
$$G\xrightarrow{\partial_k}\Omega^3G\langle 3\rangle\xrightarrow{g_k}B\g_k\xrightarrow{ev}BG.$$
First, we show properties of $\partial_k\colon G\to\Omega^3G\langle 3\rangle$.

\begin{lemma}
  \label{partial_k}
  Suppose that $G$ is $p$-regular. Then $\partial_k$ are null homotopic for all $k$ if and only if $n_l<p-1$.
\end{lemma}

\begin{proof}
  Let $\epsilon_i\colon S^{2n_i-1}\to G$ be the inclusion for $i=1,\ldots,l$, where $G\simeq\prod_{i=1}^l S^{2n_i-1}$. By \cite{L} $\partial_k$ corresponds to the Samelson product $\langle k\epsilon_1,1_G\rangle$ through the adjoint congruence $[G,\Omega^3G\langle 3\rangle]\cong[S^3\wedge G,G]$, where $\Omega^3G\langle 3\rangle$ is homotopy equivalent to the component of $\Omega^3G$ containing the basepoint. By the linearity of Samelson products, $\langle k\epsilon_1,1_G\rangle=k\langle\epsilon_1,1_G\rangle$. Thus we aim to get a condition that guarantees the triviality of $\langle\epsilon_1,1_G\rangle$. Arguing as in \cite{KaKi}, we can see that $\langle\epsilon_1,1_G\rangle$ is trivial if and only if $\langle\epsilon_1,\epsilon_i\rangle$ is trivial for all $i$. It is shown that $\langle\epsilon_1,\epsilon_i\rangle$ is trivial for all $i$ if and only if $n_l<p-1$ in \cite{KT2} when $G$ is a classical group and in \cite{HKO} when $G$ is an exceptional group. Thus the proof is complete.
\end{proof}

Since $\g_k$ is homotopy equivalent to the homotopy fiber of $\partial_k$, the following is immediate from Proposition \ref{partial_k}.

\begin{corollary}
  \label{split}
  If $n_l<p-1$ then there is a homotopy equivalence $\g_k\simeq G\times\Omega^4G\langle 3\rangle$.
\end{corollary}

Next, we consider the map $g_k\colon\Omega^3G\langle 3\rangle\to B\g_k$.

\begin{lemma}
   \label{ginjMHG}
   If $n_l<p-1$ then the restriction of $(g_k)_*$ to $MH_\text{odd}(\Omega^{3} G\langle 3\rangle)$ is an injection.
\end{lemma}

\begin{proof}
  First, for $n\geq 2$, consider the homotopy fibration
  $\Omega^{4} S^{2n+1}\to\ast\to\Omega^{3} S^{2n+1}$. We claim that the element $a_{2np-3}\in H_*(\Omega^{3} S^{2n+1})$ transgresses to a nonzero element in $H_*(\Omega^{4} S^{2n+1})$. To see this, let $E^{2}\colon S^{2n-1}\to\Omega^{2} S^{2n+1}$ be the double suspension. Let $W_{n}$ be the homotopy fibre of $E^{2}$. Then there is a homotopy fibration $\Omega S^{2n+1}\xrightarrow{\Omega E^{2}}\Omega^{3} S^{2n+1}\to W_{n}$. By \cite{CLM} this fibration induces an isomophism of Hopf algebras
  $$H_*(\Omega^{3} S^{2n+1})\cong H_*(\Omega S^{2n+1})\otimes H_*(W_{n}).$$
  In particular, since $W_{n}$ is $(2np-4)$-connected, the element $a_{2np-3}\in H_*(\Omega^{3} S^{2n+1})$ corresponds to an element $c\in H_{2np-3}(W_{n})$. On the other hand, $W_{n}$ has a single cell in dimension $2np-3$, so $c$ represents the inclusion of the bottom cell. Now consider the homotopy fibration diagram
  $$\xymatrix{\Omega^{2} S^{2n+1}\ar[r]^{\Omega^{2} E^{2}}\ar[d]&\Omega^{4} S^{2n+1}\ar[r]\ar[d]&\Omega W_{n}\ar[d]\\
  \ast\ar[r]\ar[d]&\ast\ar[r]\ar[d]&\ast\ar[d]\\
  \Omega S^{2n+1}\ar[r]^{\Omega E^{2}}&\Omega^{3} S^{2n+1}\ar[r]&W_{n}.}$$
  As $c$ represents the bottom cell in $H_{2np-3}(W_{n})$, it transgresses nontrivially to a class $d\in H_{2np-4}(\Omega W_{n})$. Since $a_{2np-3}$ maps to $c$, the naturality of the transgression implies that $a_{2np-3}$ must transgress to a nontrivial class in $H_{2np-4}(\Omega^{4} S^{2n+1})$.

  Next, consider the homotopy fibration diagram
  \begin{equation}
    \label{transdgrm}
    \xymatrix{\Omega^{4}G\langle 3\rangle\ar[r]^(.6){\Omega g_k}\ar[d]&\g_k\ar[r]^{\Omega ev}\ar[d]&G\ar[d]\\
    \ast\ar[r]\ar[d]&\ast\ar[r]\ar[d]&\ast\ar[d]\\
    \Omega^{3} G\langle 3\rangle\ar[r]^(.56){g_k}&B\g_k\ar[r]^{ev}&BG.}
  \end{equation}
  By Corollary \ref{split}, the map $\Omega g_k$ has a left homotopy inverse. In particular, $(\Omega g_k)_{\ast}$ is an injection. Since $G$ is $p$-regular for $n_l<p$, the left column in the fibration diagram is a product of the homotopy fibrations $\Omega^{4} S^{3}\langle 3\rangle\to\ast\to\Omega^{3} S^{3}\langle 3\rangle$ and, for $2\leq i\leq n_{l}$, $\Omega^{4} S^{2n_{i}-1}\to\ast\to\Omega^{3} S^{2n_{i}-1}$. Therefore, by the argument in the first paragraph of the proof, the element $a_{2n_{i}-1}\in MH_*(\Omega^{3} G\langle 3\rangle)$ transgresses to a nontrivial element in $H_{2n_{i}p-4}(\Omega^{4} G\langle 3\rangle)$. As this element injects into $H_{2n_{i}p-4}(\g_k)$, the naturality of the transgression in \eqref{transdgrm} implies that $(g_k)_*(a_{2n_{i}p-3})$ must be nontrivial.

  Finally, consider the element $\bar{a}_{2p-3}\in H_*(\Omega^{3} G\langle 3\rangle)$. It comes from an element $x\in H_{2p-3}(\Omega^{3} S^{3}\langle 3\rangle)$. The description of $H_*(\Omega^{3} S^{3}\langle 3\rangle)$ in \eqref{Thlgy} implies that $\Omega^{3} S^{3}\langle 3\rangle$ is $(2p-4)$-connected and has a single cell in dimension $2p-3$. Thus $x$ represents the inclusion of the bottom cell. Therefore, $x$ transgresses to a nontrivial element in $H_{2p-4}(\Omega^{4} S^{3}\langle 3\rangle)$, and so $\bar{a}_{2p-3}$ transgresses to a nontrivial element in $H_{2p-4}(\Omega^{4} G\langle 3\rangle)$. Since $(\Omega g_k)_*$ is an injection, the naturality of the transgression in \eqref{transdgrm} implies that $(g_k)_*(a_{2p-3})$ is nontrivial.
\end{proof}

Consider the fibration
\begin{equation}
  \label{evaluation}
  \map^*(S^4,BG)\to\map(S^4,BG)\xrightarrow{ev}BG
\end{equation}
where $ev$ is the evaluation map at the basepoint of $S^4$. Since $G$ is $p$-regular and we are localizing at the prime $p$, $S^3$ is a homotopy retract of $G$. Therefore \eqref{evaluation} is a homotopy retract of the fibration
\begin{equation}
  \label{evaluation big}
  \map^*(\Sigma G,BG)\to\map(\Sigma G,BG)\xrightarrow{ev}BG.
\end{equation}
We may identify $\map^*(\Sigma G,BG)$ with $\map^*(G,G)$ by the adjoint congruence. $\map^*(G,G)$ is an H-space by the composite of maps, and there is a fiberwise action of $\map^*(G,G)$ on $\map(\Sigma G,BG)$ given by
\[
  \map(\Sigma G,BG)\times\map^*(G,G)\to\map(\Sigma G,BG),\quad(f,g)\mapsto f\circ\Sigma g.
\]
Thus Lemma \ref{retract SS} applies to the fibration \eqref{evaluation}. The inclusion of the fibre in~\eqref{evaluation} may be identified with the coproduct of the maps
$\map^*_k(S^{4},BG)\to\map_k(S^{4},BG)$
for all $k\in\mathbb{Z}$. Equivalently, this is the coproduct of the maps
$g_k\colon\Omega^3G\langle 3\rangle\to B\g_k$ for all $k\in\mathbb{Z}$. Each $(g_{k})_{\ast}$ is injective on $MH_{odd}(\map^*(S^4,BG))$ by Lemma \ref{ginjMHG}, so the coproduct is as well. This leads to the mod-$p$ homology Serre spectral sequence for \eqref{evaluation} collapsing at the $E^{2}$-term.

\begin{proposition}
  \label{injection}
  Let $G$ be a compact connected simple Lie group of type $(n_1,\ldots,n_l)$, let $p$ be a prime, and suppose that $n_l<p-1$. Then the mod-$p$ homology Serre spectral sequence for \eqref{evaluation} is totally nonhomologous to zero.
\end{proposition}

\begin{proof}
  Let $(E^r,d^r)$ denote the mod-$p$ homology Serre spectral sequence for \eqref{evaluation}. Let $z\in H_{m}(\map^*(S^4,BG))$ be an element in the kernel of $g_*$ of least dimension, where $g\colon\map^*(S^4,BG)\to\map(S^4,BG)$ is the fiber inclusion of \eqref{evaluation}. We begin by establishing some properties of $z$.

  \textit{Property 1: $z$ is primitive}. If not, then $\overline{\Delta}(z)=\Sigma_{\alpha\in A} z'_{\alpha}\otimes z''_{\alpha}$ for some elements $z'_{\alpha},z''_{\alpha}$ of degrees $< m$ such that $\{z'_\alpha\otimes z''_\alpha\}_{\alpha\in A}$ is linearly independent, where $\overline{\Delta}$ is the reduced diagonal. The reduced diagonal is natural for any map of spaces, so $(g_*\otimes g_*)\circ\overline{\Delta}=\overline{\Delta}\circ g_*$. Now
  $$(g_*\otimes g_*)\circ\overline{\Delta}(z)=(g_*\otimes g_*)(\Sigma_{\alpha} z''_{\alpha}\otimes z''_{\alpha})=\Sigma_{\alpha}g_*(z'_{\alpha})\otimes g_*(z''_{\alpha})$$
  is a sum of linearly independent elements since $z'_{\alpha},z''_{\alpha}$ are of degrees $< m$ while the element of least degree in $\mathrm{Ker}\,g_*$ is of degree $m$. On the other hand,
  $$\overline{\Delta}\circ g_*(z)=\overline{\Delta}(0)=0.$$
  This contradiction implies that $\overline{\Delta}(z)$ must be $0$;
  that is, $z$ is primitive.

  \textit{Property 2: $\mathcal{P}^n_*(z)=0$ for every $n\ge 1$}. Suppose $\mathcal{P}^n_{\ast}(z)=y$ for some $n\ge 1$, where $y$ is nonzero. The (dual) Steenrod operations are natural for any map of spaces, so $g_*\mathcal{P}^n_*=\mathcal{P}^n_*g_*$. Now
  $$g_*\mathcal{P}^n_{\ast}(z)=g_*(y)$$
  is nonzero, since $|y|<m$ while the element of least degree in $\mathrm{Ker}\,g_*$ is of degree $m$. On the other hand,
  $$\mathcal{P}^n_{\ast}g_*(z)=\mathcal{P}^n_*(0)=0.$$
  This contradiction implies that $\mathcal{P}^n_{\ast}(z)=0$ for every $n\ge 1$.

  \textit{Property 3: $\beta^{r}(z)=0$ for every $r\geq 1$}. The reasoning is exactly as in the proof of Property 2.

  \textit{Property 4: $z$ is in the image of the transgression for the mod-$p$ homology Serre spectral sequence for the homotopy fibration \eqref{evaluation}}. By assumption, the first nontrivial differential $d^r$ in the mod-$p$ homology Serre spectral sequence for \eqref{evaluation} satisfies the properties of Lemma~\ref{retract SS}. In particular, $d^r$ is determined by how it acts on the elements in $H_*(BG)$. The $E^{2}$-term of the spectral sequence is $H_*(BG)\otimes H_*(\map^*(S^4,BG))$. Since $z$ is an element of least degree in the kernel of $(g_k)_*$ and is of degree $m$, the map $(g_k)_*$ is an injection in degrees $<m$. Thus the spectral sequence is totally non-homologous to zero in this degree range, implying that it collapses at $E^{2}$. In particular, the differential $d^r$ on an element of $H_*(BG)$ of degree $\leq m$ is zero, and so as $d^r$ satisfies the properties of Lemma \ref{retract SS}, $d^r(x\otimes y)=0$ for $x\in H_t(BG)$ with $t\leq m$ and $y\in H_*(\map^*(S^4,BG))$. Hence, for degree reasons, if $z$ is in the image of $d^r$ then the only possibility is that $z=d^r(x)$ and $r=m+1$ for some $x\in H_{m+1}(BG)$. That is, $z$ is in the image of the transgression. On the other hand, as $z\in\mathrm{Ker}\,g_*$, it cannot survive the spectral sequence and so must be hit by some differential.

  \textit{Property 5: $z$ has odd degree}. By Property 4, in the mod-$p$
  homology Serre spectral sequence for \eqref{evaluation}, we have $z=d^{m+1}(x)$ for some $x\in H_{m+1}(BG)$. As $H_*(BG)$ is concentrated in even degrees, the degree of $z$ must be odd.

  Let's examine the consequences of Properties 1 to 5. Collectively, Properties 1 to 3 imply that $z\in MH_*(\map^*(S^4,BG))$. Property 5 then implies that $z\in MH_{odd}(\map^*(S^4,BG))$. But $g_*$
  is an injection on $MH_{odd}(\map^*(S^4,BG))$ as we saw above, implying in turn that $z\notin\mathrm{Ker}\,g_*$, a contradiction. Thus $g_*$ is an injection, and this
  completes the proof.
\end{proof}

We are ready to prove Theorem \ref{Bgaugehlgy}.

\begin{proof}
  [Proof of Theorem \ref{Bgaugehlgy}]
  By Proposition \ref{injection}, the mod-$p$ homology Serre spectral sequence for \eqref{evaluation} collapses at the $E^2$-term. Thus, as \eqref{evaluation} is the fiberwise coproduct of the fibrations
  $\map^*_k(S^{4},BG)\to\map_k(S^{4},BG)\xrightarrow{ev} BG$
  in \eqref{eval_k} for all $k\in\mathbb{Z}$, it follows from Lemma \ref{collapse component} that the mod-$p$ homology Serre spectral sequence for each fibration collapses at the $E^2$-term. Equivalently, the mod-$p$ homology Serre spectral sequence for
  $\Omega^3 G\langle3\rangle\to B\g_k\xrightarrow{ev} BG$
  collapses at the $E^{2}$-term for all $k\in\mathbb{Z}$.
\end{proof}

We calculate the Poincar\'e series of the mod-$p$ homology of $B\g_k$.
Let $P_t(X)$ denote the Poincar\'e series of the mod-$p$ homology of a space $X$. Let
\begin{align*}
  P(t)&=\prod_{k\ge 0}(1+t^{2(p-1)p^k-1})\prod_{k\ge 1}\frac{1}{1-t^{2(p-1)p^k-2}}\\
  Q_n(t)&=\prod_{k\ge 1,j\ge 0}(1+t^{2(np^k-1)p^j-1})\prod_{k\ge 1,j\ge 1}\frac{1}{1-t^{2(np^k-1)p^j-2}}\prod_{k\ge 0}\frac{1}{1-t^{2n^k-2}}.
\end{align*}

\begin{corollary}
  \label{Poincare_series}
  Under the same hypothesis of Theorem \ref{Bgaugehlgy}, the Poincar\'e series of the mod-$p$ homology of $B\g_k$ is given by
  $$P_{t}(B\g_k)=P(t)Q_{p-1}(t)\prod_{i=2}^l Q_{n_i-1}(t)\prod_{i=1}^l\frac{1}{1-t^{2n_i}}.$$
\end{corollary}

\begin{proof}
By Theorem \ref{Bgaugehlgy}, $P_t(B\g_k)=P_t(BG)P_t(\Omega^3G\langle 3\rangle)$. Since $n_l<p-1$, the mod-$p$ cohomology of $BG$ is a polynomial algebra with generators in dimension $2n_1,\ldots,2n_l$, and so $P_t(BG)=\prod_{i=1}^l\frac{1}{1-t^{2n_i}}$. By \eqref{product_decomp}, $P_t(\Omega^3G\langle 3\rangle)=P_t(\Omega^3S^3\langle 3\rangle)\prod_{i=2}^l P_t(\Omega^3S^{2n_i-1})$. By \eqref{CLMcalc}, \eqref{Thlgy} and \eqref{CLMcalc_2p-1}, $P_t(\Omega^3S^3\langle 3\rangle)=P(t)Q_{p-1}(t)$ and $P_t(\Omega^3S^{2n_i-1})=P_{n_i-1}(t)$, completing the proof.
\end{proof}

%%%%% Section 5 %%%%%

\section{The mod-$3$ homology of $SU(2)$-gauge groups}
\label{sec:SU2}

The mod-$p$ homology of $B\g_k$ for $G=SU(2)$ and $p\geq 5$ is given by Theorem \ref{Bgaugehlgy}. Since $SU(2)$-gauge groups are pivotal in Donaldson Theory, in this section we expand beyond the statement of Theorem \ref{Bgaugehlgy} by also calculating the mod-$3$ homology.

Since there is a homeomorphism $SU(2)\cong S^{3}$, we phrase what follows in terms of $S^{3}$. By \eqref{Serrefib} there is a homotopy fibration sequence
$$S^{3}\xrightarrow{\partial_k}\Omega^{3} S^{3}\langle 3\rangle\xrightarrow{g_k}B\g_k\xrightarrow{ev}BS^{3}.$$
By \cite{K} we have the following.

\begin{lemma}
   \label{Kono}
   The map $\partial_1\colon S^{3}\to\Omega^{3} S^{3}\langle 3\rangle$ has order $12$.
\end{lemma}

In the proof of Lemma \ref{partial_k} it is shown that $\partial_k=k\circ\partial_1$. In particular, if we localize at $p=3$ then $\partial_k$ has order $(k,3)$. Observe that for $p=3$, the description of $H_*(\Omega^{3} S^{3}\langle 3\rangle)$ in \eqref{Thlgy} implies that $\Omega^{3} S^{3}\langle 3\rangle$ is $2$-connected with a single cell in dimension $3$.

\begin{lemma}
   \label{partialkimage}
   If $p=3$ and $(k,3)=1$ then the map $\partial_k\colon S^3\to\Omega^{3} S^{3}\langle 3\rangle$ induces an injection in mod-$3$ homology.
\end{lemma}

\begin{proof}
  By Lemma \ref{Kono} and the above observation, $\partial_k$ is nontrivial. Thus it represents a nontrivial element in $\pi_{3}(\Omega^{3} S^{3}\langle 3\rangle)\cong\pi_{6}(S^{3}\langle 3\rangle)\cong\pi_{6}(S^{3})$. By \cite{T}, the $3$-component of $\pi_{6}(S^{3})$ is isomorphic to $\Z/3\Z$. Then since $\Omega^3S^3\langle 3\rangle$ is $2$-connected, the Hurewicz map $\pi_3(\Omega^3S^3\langle 3\rangle)\otimes\Z/3\Z\to H_*(\Omega^3S^3\langle 3\rangle)$ is an isomorphism, where we take homology with mod-$3$ coefficients. By the naturality of Hurewicz maps there is a commutative diagram
  $$\xymatrix{\pi_3(S^3)\otimes\Z/3\Z\ar[r]^(.45){(\partial_k)_*}\ar[d]&\pi_3(\Omega^3S^3\langle 3\rangle)\otimes\Z/3\Z\ar[d]\\
  H_3(S^3)\ar[r]^(.45){(\partial_k)_*}&H_3(\Omega^3S^3\langle 3\rangle)}$$
  where the vertical maps are the Hurewicz maps. We saw that the top and the right maps are isomorphisms. The left arrow is an isomorphism by the Hurewicz theorem. Thus the bottom arrow is an isomorphism, completing the proof.
\end{proof}

Next, by \eqref{Thlgy} together with \eqref{CLMcalc} and \eqref{CLMcalc_2p-1}, $H_{3}(\Omega^{3} S^{3}\langle 3\rangle)\cong\Z/3\Z$ and if $x_{3}$ is its generator, then $\Lambda(x_{3})$ is a tensor product factor of $H_*(\Omega^{3} S^{3}\langle 3\rangle)$.

\begin{proof}
  [Proof of Theorem \ref{p=3}]
  Consider the homotopy fibration $S^{3}\xrightarrow{\partial_k}\Omega^{3} S^{3}\langle 3\rangle\xrightarrow{g_k}B\g_k$. By Lemma \ref{partialkimage}, $(\partial_k)_*$ is an injection. Therefore the homology Serre spectral sequence for this homotopy fibration is totally non-homologous to zero, implying that it collapses at the $E^{2}$-term, so $H_*(\Omega^{3} S^{3}\langle 3\rangle)\cong H_*(S^{3})\otimes H_*(B\g_k)$. Since $H_*(S^{3})\cong\Lambda(y_{3})$ for $|y_3|=3$ such that $y_3$ corresponds to $x_3$ and as observed above, $\Lambda(x_{3})$ is a tensor product factor of $H_*(\Omega^{3} S^{3}\langle 3\rangle)$, we obtain the isomorphism asserted in the statement of the Theorem.
\end{proof}

%%%%% Reference %%%%%

\end{document}